\documentclass{article}

\newenvironment{proofof}[1]{\begin{trivlist}\item[\hskip\labelsep{\it
Proof~of~{#1}.\ }]}{\hspace*{\fill} {\qedsymbol}\end{trivlist}}

\newcommand{\cons}{\ensuremath{{\sf Con}}}
\renewcommand{\Pr}{\ensuremath{{\sf Pr}}}

\newcommand{\code}[1]{\ensuremath{\ulcorner{#1}\urcorner}}
\newcommand{\numeral}[1]{\ensuremath{\underline{{#1}}}}

\newcommand{\bewijs}[1]{\ensuremath{\vdash_{{#1}}}}
\newcommand{\oo}[1]{\ensuremath{\mathcal{O}({#1})}}
\newcommand{\ohm}[1]{\ensuremath{\Omega({#1})}}
\newcommand{\isig}[1]{{\ensuremath {\mathrm{I}\Sigma_{#1}}}\xspace}
\newcommand{\sonetwo}{{\ensuremath{{\sf S^1_2}}}\xspace}

\newcommand{\eqref}[1]{(\ref{#1})}

\usepackage{amssymb}
\usepackage{amsfonts}
\usepackage{latexsym}
\usepackage{xspace}

\usepackage{amsthm}

\theoremstyle{plain}
\newtheorem{theorem}{Theorem}[section]
\newtheorem{lemma}[theorem]{Lemma}
\newtheorem{proposition}[theorem]{Proposition}
\newtheorem{definition}[theorem]{Definition}

\theoremstyle{remark}

\theoremstyle{question}
\newtheorem{question}[theorem]{Question}
\newtheorem{fact}[theorem]{Fact}
\newtheorem{conjecture}[theorem]{Conjecture}
\newtheorem{observation}[theorem]{Observation}

\title{Propositional proof systems and fast consistency provers}
\date{2006}

\author{Joost J. Joosten}

\begin{document}

\maketitle\footnote{This paper is partly written while employed at the Mathematical Institute of the Academy of Sciences of the Czech Republic in Prague. The stay in Prague was also financed by the Netherlands Organization for Scientific Research (NWO).}


\begin{abstract}
A fast consistency prover is a consistent poly-time axiomatized theory that has short proofs of the finite consistency statements of any other poly-time axiomatized theory. 
Kraj\'\i\v{c}ek and Pudl\'ak proved in \cite{KrajicekPudlak:propproofs} that the existence of an optimal propositional proof system is equivalent to the existence of a fast consistency prover. It is an easy observation that ${\sf NP}={\sf coNP}$ implies the existence of a fast consistency prover. The reverse implication is an open question.

In this paper we define the notion of an unlikely fast consistency prover and prove that its existence is equivalent to 
${\sf NP}={\sf coNP}$.

Next it is proved that fast consistency provers do not exist if one considers RE axiomatized theories rather than theories with an axiom set that is recognizable in polynomial time. 
\end{abstract}

\section{Introduction}
There are many interesting relations between computational complexity and arithmetic. In this paper we shall focus on one such relation that involves length of proofs of finite consistency statements. In particular, we shall study \emph{fast consitstency provers}. Basically a fast consistency prover, a facop for short, is a certain theory $S$
that has short proofs of the finite consistency statements of any other certain theory $T$. We shall see precise definitions shortly.

Kraj\'\i\v{c}ek and Pudl\'ak  proved in \cite{KrajicekPudlak:propproofs} that if there is no fast consistency prover, then ${\sf NP}\neq {\sf coNP}$. We shall plead that it is very unlikely that a facop can exist. It is an open question whether the existence of a facop is actually equivalent to ${\sf NP} = {\sf coNP}$. In Section 
\ref{section:UfacopsEnUppses} we shall define the notion of an 
\emph{unlikely fast consistency prover}, a ufacop for short, and show that the existence of a ufacop is equivalent to ${\sf NP} = {\sf coNP}$. \medskip

Before we shall plead that the existence of a facop is unlikely, let us first specify some definitions. In this paper, we shall always mean by the 
\emph{length} of a proof the number of symbols occurring in it. If $S$ is a theory, we shall denote by $S\vdash_n\varphi$ that $\varphi$ is provable in $S$ by a proof whose length does not exceed $n$. We shall denote the formalization/arithmetization of this statement by ${\sf Pr}_S(n,\ulcorner \varphi \urcorner)$. For those familiar with formalized provability it is good to stress that there is a logarithm involved here, that is,
\[
{\sf Pr}_S(x,y) \ := \ \exists \pi\ (\, |\pi|{\leq} x \ \wedge \ 
{\sf Proof}_S(\pi,y)). 
\]
Here ${\sf Proof}_S(x,y)$ is a natural arithmetization of ``$x$ is the G\"odel number of a proof in $S$ of a formula with G\"odel number $y$''.  
All theories considered in this paper will be first order theories of some minimal strength which are sound and hence consistent. With 
${\sf Con}_T(x)$ we shall denote $\neg{\sf Pr}_T(x,\ulcorner 0=1\urcorner)$.

If a theory $T$ has a set of axioms which is decidable/recognizable in polynomial time, we shall speak of a \emph{poly-time theory}. If $\varphi$ is provable in $S$, we shall denote by $||\varphi||_S$ the length of the shortest proof in $S$ of $\varphi$. If $n$ is a natural number, we shall denote by $\underline{n}$ denote its efficient (dyadic) numeral. We are now ready to give the definition of a fast consistency prover.

\begin{definition}\label{definition:facop}
A \emph{fast consistency prover} (facop) is a consistent poly-time theory $S$ such that for any other consistent poly-time theory $T$ there is a polynomial $p$ such that 
\[
||{\sf Con}_T(\underline{n})||_S \leq p(n).
\]
\end{definition}
Now, why is it hard to believe in the existence of a facop? First of all, let us remark that a facop is well defined. As, by our assumption, $T$ is consistent, we first remark that ${\sf Con}_T(\underline{n})$ is indeed true. But 
${\sf Con}_T(\underline{n})$ is also provable in $S$. This is because there are at most $2^n$ many proofs whose length are below $n$. So, in $S$ all this many proofs can be listed and combined with the observation that none of these proofs is a proof of $0=1$. 

This brings us directly to the question of how a facop could possibly exist. For, if $T$ is completely arbitrary, what else can $S$ do than just give the list of all possible proofs and remark that none is a proof of $0=1$.  For $T$ weaker than $S$ it seems conceivable that $S$ can do some smart tricks and summarize this long list. But, if $T$ is a lot stronger than $S$ it seems very strange that $S$ would have a short way of proving the finite consistency statements of $T$.

It is good to realize here that the poly-time axiomatizability is not directly saying anything about the proof strength of a theory. For example, a poly-time theory may contain an axiom 
${\sf Con}({\sf ZFC}+ \mbox{``there exists a superhuge cardinal''})$ or any other consistent large cardinal assumption that makes your head spin round.

But it seems hard to relate proof strength to the length of proofs of finite consistency statements. In Section \ref{section:FacopsAndOppses} we shall define a hypothetical facop $S$ (in the proof of Theorem \ref{theorem:facopiffopps}). This $S$ consists of a very weak fragment of arithmetic plus the assumption that some hypothetical propositional proof system only proves tautologies. All these ingredients seem to have little to do with proof strength.

The most tempting way to prove the non-existence of facops is by using diagonalization, that is, by using fixed points. In 
Section \ref{section:ReFacopsDoNotExist} we set up such an approach for RE-facops. An RE-facop is obtained by replacing ``poly-time'' in Definition 
\ref{definition:facop} by ``RE''. In particular we show that RE-facops do not exist. 

It is good to mention here a result by Pudl\'ak. In \cite{Pudlak:selfconsisI} and \cite{Pudlak:selfconsisII} he proved that for a large class of theories $T$, the $||{\sf Con}_T(\underline{n})||_T$ can be bounded by a polynomial in $n$.

In addition it is good to mention that questions about the length of proofs of finite consistency statements have an interest on themselves, not related to computational complexity. In particular they have a close relation to foundations of mathematics and possible partial realizations of Hilbert's program.

\section{Preliminaries}
In this section we provide the basic definitions that are needed further on in the paper. Probably it is best to just skip this section and turn to it if necessary. 

As mentioned in the introduction, in this paper we shall study a relation between arithmetic and computational complexity. By choosing/tailoring the arithmetic language in the right way there are straightforward correspondences. 

For this reason we shall in this paper always consider theories in the language of bounded arithmetic (see e.g.\ \cite{BusH2}). This language is an extension to the basic language of arithmetic in that it contains symbols for the binary logarithm $|x|$ and for the function $\omega_1(x)$. Here 
$\omega_1(x)= 2^{|x|^2}$. From now on, all arithmetic formulas in this paper will be in the language of bounded arithmetic. 

We shall employ the usual hierarchy of bounded formulas in this paper. Thus, 
$\Delta^b_0$ is the class of of formulas (in the language of bounded arithmetic) which contains all open formulas, and which which is closed under all boolean connectives and under sharply bounded quantification. Here sharply bounded quantification is quantification of the form $\forall \, x{\leq}|t|$ or $\exists \, x{\leq}|t|$. Here, $t$ is some term in the language of bounded arithmetic that does not contain $x$ as a variable.

Next, we define $\Delta_0^b = \Sigma_0^b = \Pi^b_0$. The 
$\Sigma_{i+1}^b$ formulas are those obtained by closing off the 
$\Pi^b_i$ formulas under bounded existential quantification, boolean connectives and sharply bounded universal and existential quantification. The 
$\Pi^b_{i+1}$ formulas are defined dually. Bounded quantification is
quantification of the form $\forall \, x{\leq}t$ or $\exists \, x{\leq}t$. Again, $t$ is some term in the language of bounded arithmetic that does not contain $x$ as a variable.

The language is chosen in such a way that there is a close correspondence between computational complexity classes and definable sets. We say that a formula $\alpha(x)$ defines a set of natural numbers $A$ if $x\in A \ \ \Leftrightarrow \ \ \mathbb{N}\models \alpha(x)$. It is not too hard to see the following correspondences.
\[
\begin{array}{lll}
A \mbox{ is $\Delta_0^b$ definable} & \Rightarrow & A \in {\sf P}\\
A \mbox{ is $\Sigma_1^b$ definable} & \Leftrightarrow & A \in {\sf NP}\\
A \mbox{ is $\Pi_1^b$ definable} & \Leftrightarrow & A \in {\sf coNP}\\
\end{array}
\]
This correspondence is pretty straightforward and can be easily continued through all the bounded formula complexity classes by using oracles. Note that for the complexity class ${\sf P}$ we have no equivalence. In order to get an equivalence some non trivial mathematics has to be applied. In particular, as a consequence of Buss' Witnessing Theorems we have the following.
\[
\begin{array}{lll}
A \mbox{ is $\Delta_1^b(S^1_2)$ definable} & \Leftrightarrow & A \in {\sf P}\\
\end{array}
\]
A formula is $\Delta_1^b(S^1_2)$ if it is in $S_2^1$ equivalent to both a 
$\Sigma^b_1$ formula and a $\Pi_1^b$ formula. Here, $S_2^1$ is a pretty weak arithmetic theory with other than the defining axioms of the symbols of the language of arithmetic a weak form of induction for $\Sigma_1^b$ formulas. We refer the reader to \cite{BusH2} for details. In the rest of this paper we shall often speak of $\Delta_1^b(S^1_2)$ theories instead of poly-time theories.
\medskip

We have seen one correspondence between arithmetic and complexity by the above definability results. Another correspondence goes via propositional proof systems as introduced by Cook and Reckhow in \cite{CookReckhow:proofs}. Let us briefly give the basic definitions and facts here.
\begin{definition}\label{defi:pps}
A propositional proof system, a pps for short, is a poly-time mapping from the set of all strings onto the set of all tautologies.
\end{definition}

All propositional proof systems we know, be it natural deduction, Gentzen or whatever, can be seen as a pps by mapping a string of syntax that is not a proof in this particular system to the tautology 1 and by mapping a string that is a proof to the tautology it proves. Checking whether a string is a proof or not is for all known proof systems poly-time (even cubic time would suffice, as to get parsing of contex-free grammars).

An easy correspondence between propositional proof systems and complexity is given by Theorem \ref{theorem:CookReckhow} which is due to Cook and Reckhow and relates the existence of so-called super proof systems to ${\sf NP}={\sf coNP}$.

\begin{definition}
A pps $P$ is called \emph{super} if there is a polynomial $p$ such that 
\[
\forall^{\sf Taut} \tau \, \exists \,  |\pi | {<} p(|\tau |) \ P(\pi,\tau).
\]
\end{definition}

\begin{theorem}\label{theorem:CookReckhow}
${\sf NP}={\sf coNP}$ if and only if there exists a super pps.
\end{theorem}

It is important to compare different pps's to each other in terms of the 
\emph{size}, i.e. \emph{length} of the proofs, which is nothing but the total number of symbols occurring in it. If $\pi$ is a proof, we shall denote its length by $|\pi |$. This suggests a logarithmic relation which is good: the length of a string over a finite alphabet is, under efficient coding, linear in the binary logarithm of the code of that string.

\begin{definition}
Let $P$ and $Q$ be pps's and let $f$ be a function. We define:\\
\begin{itemize}
\item[-]
$P\geq_{f(x)}Q \ \ := \ \ Q(\pi,\tau) \to \exists  \pi'\ (|\pi' | \leq f (|\pi|) \wedge P(\pi', \tau))$

\item[-]
$P \geq Q \ \ := \ \ \mbox{ for some polynomial $p$, } \ P \geq_{p(x)}Q$. In this case we say that $P$ \emph{polynomially simulates} $Q$.

\item[-]
$P\equiv Q \ \ := \ \ (P\geq Q) \ \ \& \ \ (Q\geq P)$ 
\end{itemize}

\end{definition}

Throughout this paper we shall assume that our binding polynomials are monotone increasing which is not an essential assumption, but makes the proofs easier.

In all known propositional proof systems it holds that the tautology is at most as long as the proof of the tautology. This does not follow from the general definition of a pps. However, the following lemma tells us that we, for many purposes, may assume without loss of generality that, indeed, the proof of a tautology is at most as long as the tautology.

\begin{lemma}\label{lemma:tautsmallerthanproof}
For every pps $P$, there is a pps $P'$ such that $P'\equiv P$ and $P'(\pi,\tau)\to |\tau|\leq |\pi|$.
\end{lemma}

\begin{proof}
From $P$ we define $P'$ as 
\[
P'(\pi',\tau) : \Leftrightarrow 
[\pi' = (\pi{\smallfrown}\tau)] \wedge P(\pi,\tau)
\]
where $\smallfrown$ denotes concatenation. Clearly $P\geq_xP'$. If now $P(\pi,\tau)$, we can retrieve $\tau$ from $\pi$ in poly-time, so, certainly $|\tau| \leq p'(\pi)$ for some polynomial $p'$. Consequently, $|\pi{\smallfrown}\tau|\leq p(\pi)$ for some polynomial $p$ and $P'\geq P$.
\end{proof}

We shall often identify a pps and its $\Delta^b_1(S_2^1)$ definition in bounded arithmetic.
The following definition is central to the rest of this paper.
\begin{definition}
A pps $P$ is an \emph{optimal propositional proof system}, an opps for short, if 
$P\geq Q$ for any propositional proof system $Q$.
\end{definition}
It is easy to see that a pps is optimal whenever it is super. Thus via Cook and Reckhow's theorem (Theorem \ref{theorem:CookReckhow}) we get that 
\[
{\sf NP}={\sf coNP} \Rightarrow \mbox{there exists an opps}.
\]
It is an open question whether the converse implication holds.

\section{Fast consistency provers and optimal propositional proof systems}\label{section:FacopsAndOppses}

Kraj\'\i\v{c}ek and Pudl\'ak proved that the existence of an opps is equivalent to the existence of a facop. In this section we shall give a self contained version of this proof. The next section will then build forth on this proof to obtain a similar result.

\begin{theorem}\label{theorem:facopiffopps}
\ \ $\exists  \mbox{\ facop}\ \ \Longleftrightarrow \ \ \exists \mbox{\ opps}$
\end{theorem}

Before we can present a proof of this theorem, we should first mention some results involving length of proofs and discuss some coding machinery.

\begin{definition}
A relation $R$ is polynomially numerable in a theory $T$ if for some polynomial $p$ and some formula $\rho$ we have that
\[
R(x) \ \Leftrightarrow  \ T\vdash \rho (\underline{x}) \ 
\Leftrightarrow \ T \vdash_{p(|x|)}\rho (\underline{x}) .
\]
\end{definition}
It is good to stress here that $\underline{x}$ denotes the efficient numeral of $x$ so that the length of $\underline{x}$ is logarithmic in $x$.
\begin{theorem}\label{theorem:NPenumerable}
The following are equivalent.
\begin{enumerate}
\item
$R\in {\sf NP}$

\item
$R$ is polynomially numerable in robinson's arithmetic $\sf R$
\end{enumerate}
\end{theorem}

\begin{proof}
A proof of this theorem can be found in \cite{Pudlak98}. The $\Leftarrow$ is easy and actually holds for any poly-time axiomatized theory $T$.

The $\Rightarrow$ direction goes by coding of computations on Turing machines. To get really as low as $\sf R$ here, some additional tricks with definable cuts are needed. 
\end{proof}

If $R\in {\sf NP}$, it is definable by a $\Sigma$ (even $\Sigma_1^b$) formula $\rho$ and thus, for any (sound, poly-time axiomatized) theory $T$ extending $Q$ we have that 
\[
\begin{array}{ll}
\mbox{$R$ is polynomially numerable by $\rho$ in $Q$ }& \Leftrightarrow\\
\mbox{$R$ is polynomially numerable by $\rho$ in $T$ }& \ \\
\end{array}
\]
Having this in mind, we can consider provable $\Sigma_1^b$-completeness as expressed in the next theorem as a formalization of the above 
(Theorem \ref{theorem:NPenumerable} plus remark).

\begin{theorem}\label{theorem:sigmacomplete}
Let $T$ be a $\Delta_1^b(S_2^1)$ theory extending $S_2^1$. For every 
$\Sigma_1^b$ formula $\sigma (x)$, there is a polynomial $p$ such that 
\[
S_2^1 \vdash \forall x\ 
(\sigma (x) \to {\sf Pr}_T(p(|x|), \ulcorner \sigma (\dot{x })\urcorner)).
\]
\end{theorem}

Coding of syntax for propositional logic in arithmetic can be done in a standard way. If $\tilde{a}$ is a sequence of zeroes and ones and $\varphi(\vec{p})$ a propositional formula, there is a $\Delta_1^b(S^1_2)$ formula saying that 
$\varphi(p_i/(\tilde{a})_i)$ evaluates to one. We shall write
\[
\tilde{a}\models \varphi.
\]
There is a $\Pi^b_1$ formula ${\sf Taut}(\tau)$ saying that ${\sf Taut} (p_0,\ldots,p_n)$ is a tautology. This formula is defined as 
\[
{\sf Taut}(\tau) \ := \ \forall\, |a|{\leq}(n{+}1)\  \tilde{a}\models \tau .
\]
If $Q$ is a pps, we shall denote by ${\sf RFN}(Q)$ the formalized reflection over $Q$, that is, the following $\forall \Pi_1^b$ formula saying that all provable formulas are true.
\[
{\sf RFN}(Q)\ :=\ \forall \tau\ (\exists \pi\ Q(\pi,\tau)\to {\sf Taut}(\tau))
\]

In a sense, we can even code arithmetic (and a fortiori syntax) into propositional logic. This is expressed in the following lemma.

\begin{lemma}\label{lemma:coding}
There exists a translation of $\Pi_1^b$-formulas $\varphi(x)$ in the language of bounded arithmetic into series of propositional formulas $||\varphi||^m$ such that: 
\begin{enumerate}
\item 
The translation preserves the structure of $\varphi$. E.g., 
$||\chi \wedge \psi||^m=||\chi||^m \wedge ||\psi||^m$, where $\chi$ and $\psi$ are subformulas of $\varphi$.

\item 
The translation $||\varphi||^m$ contains variables $\vec{q}$ and 
$p_0,\ldots, p_m$.

Instead of writing in the arithmetical correct way that 
$\varphi$ is a tautology when the binary representation $\tilde{a}$ is substituted for the $p_i$, i.e., 
\[
{\sf Taut}(||\varphi||^m(\vec{q}, \vec{p}/\tilde{a})),
\] 
we shall use the following shorthand notation.
\[
\tilde{a}\models ||\varphi||^m
\]

\item \label{lemma:coding:item:adequate}
The translation is provably adequate in the following sense.
\[
S_2^1 \vdash \forall\,  |a| {\leq} (m{+}1)\ (\varphi (a) \leftrightarrow 
\tilde{a}\models ||\varphi||^m) 
\]

\item \label{lemma:coding:item:bound}
The translation is short in the following sense.
For each $\varphi$ there exists a polynomial $p$ such that
\[
|(||\varphi||^m)|\leq p(m) .
\]

\end{enumerate}

\end{lemma}

Note that $||\varphi||^m$ and $||\varphi||_T$ denote two completely different things. We are confident that the reader can keep them apart.
Now that all coding machinery has been discussed, we are ready to present a proof of Theorem \ref{theorem:facopiffopps}.

\begin{proofof}{Theorem \ref{theorem:facopiffopps}}
In this proof, we shall denote by $T\vdash_{\star} \varphi$ the statement that $\varphi$ is provable in $T$ by a proof whose length is bounded by some polynomial on the (length of) the parameters of $\varphi$. Sometimes we shall have to specify the parameters to keep the intended reading clear.
\medskip

``$\Rightarrow$'' We repeat the proof from \cite{KrajicekPudlak:propproofs}, and \cite{Krajicek:book} (Theorem 14.1.4). Let $S$ be a facop. We define $P$ and show that $P$ is an opps.
\[
\begin{array}{ll}
P(\pi, \tau)\ := & {\sf Proof}_S(\pi, {\sf Taut}(\tau)) \ \mbox{ or }\\
\ & ( \tau =1 \mbox{ and $\pi$ is not a proof in $S$ of ${\sf Taut}(\tau')$ for any $\tau'$)}
\end{array}
\]
To see that $P$ is an opps we fix some arbitrary $Q$ and consider some
$\pi$ and $\tau$ such that $Q(\pi,\tau)$. By Theorem \ref{theorem:NPenumerable} we get that
\begin{equation}\label{equation:bewijsbaarinQ}
S^1_2 \vdash_{\star} Q(\pi,\tau).
\end{equation}
We now define $T_Q := S^1_2+{\sf RFN}(Q)$.
Clearly, by \eqref{equation:bewijsbaarinQ} and by ${\sf RFN}(Q)$ we get
$T_Q\vdash_{\star} {\sf Taut}(\tau)$. Here, the $\star$ is still dependent on $|\pi|$. Once more, by Theorem \ref{theorem:NPenumerable} we get for some polynomial $p$ that
\begin{equation}\label{equation:TQbewijst}
S\vdash_{\star} {\sf Pr}_{T_Q}(p(|\pi|), \ulcorner {\sf Taut}(\tau)\urcorner) .
\end{equation}
By Theorem \ref{theorem:sigmacomplete}, for some polynomial $p'$ we have that
\begin{equation}\label{equation:Tautreflectie}
S\vdash \neg {\sf Taut}(\tau) \to {\sf Pr}_{T_Q}(p'(|\tau|), \ulcorner\neg {\sf Taut}(\tau)\urcorner) .
\end{equation}
Combining \eqref{equation:TQbewijst} and \eqref{equation:Tautreflectie} we get for some polynomial $q$ that 
\[
S\vdash_{\star} \neg {\sf Taut}(\tau) \to {\sf Pr}_{T_Q}(q(|\pi|), \ulcorner 0=1\urcorner)
\]
or equivalently
\[
S\vdash_{\star} {\sf Con}_{T_Q}(q(|\pi|)) \to {\sf Taut}(\tau).
\]
As $S$ is a facop, we get $S\vdash_{\star}{\sf Con}_{T_Q}(q(|\pi|))$, whence
$S\vdash_{\star}{\sf Taut}(\tau)$ and thus, $P\vdash \tau$ by a proof whose length is polynomial in $|\pi|$.
\medskip

"${\bf{\Leftarrow}}$" Let $P$ be an opps. We define
\[
S\ := \  S^1_2 + {\sf RFN}(P)
\]
and shall prove that $S$ is a facop. So, let $T$ be some $\Delta_1^b(S^1_2)$ theory. We should see that ${\sf Con}_T(\underline{x})$ has short proofs in $S$. For this purpose we define $Q$ as follows.
\[
Q\ := \ P + \{ ||{\sf Con}_T(|x|)||^m\mid m < \omega \}
\]
Note that, due to the logarithm, ${\sf Con}_T(|x|)$ is indeed a $\Pi^b_1$ formula. As $P$ is an opps, we get by 
Item \ref{lemma:coding:item:bound} of Lemma \ref{lemma:coding} that
\[
P\vdash_{\star}||{\sf Con}_T(|x|)||^m .
\]
Here, the $\star$ refers to polynomial in $m$. By Theorem \ref{theorem:NPenumerable} we get that 
\[
S^1_2 \vdash_{\star} \exists y\ P(y,||{\sf Con}_T(|x|)||^m)
\]
and consequently,
\[
S\vdash_{\star} {\sf Taut}(||{\sf Con}_T(|x|)||^m).
\]
In particular
\[
S\vdash_{\star} \forall \, |a|{\leq}(m{+}1) \ \tilde{a}\models||{\sf Con}_T(|x|)||^m.
\]
For $a=2^m$ we get via Item \ref{lemma:coding:item:adequate} from Lemma
\ref{lemma:coding} that 
\[
S\vdash_{\star}{\sf Con}_T(\underline{m}).
\]
In other words, for some polynomial $p$,
\[
||{\sf Con}_T(\underline{m})||_S \leq p(m)
\]
and, indeed, $S$ is a facop.
\end{proofof}

\section{Unlikely fast consistency provers and unlikely propositional proof systems}\label{section:UfacopsEnUppses}

In the previous section we have studied facops. As a direct consequence of 
Theorem \ref{theorem:facopiffopps} we get that
\[
{\sf NP}={\sf coNP} \Rightarrow \mbox{there exists a facop}.
\]
As mentioned before, the converse implication is an open question. In this section we shall define a unlikely fast consistency prover, a ufacop for short, which is a particular sort of facops. We shall then prove that the existence of a ufacop is equivalent to 
${\sf NP}={\sf coNP}$.

In order to prove this, we shall have to employ a slightly different definition of a pps. However, in the light of Lemma \ref{lemma:tautsmallerthanproof} this alteration is not really essential.

\begin{definition}
A pps $P$ is a poly-time mapping from the set of all strings onto the set of all tautologies such that $P(\pi,\tau)$ implies $|\tau|\leq |\pi|$. 
\end{definition}

To the best of our knowledge, there is no theorem concerning pps's that does not remain valid under this new definition.

\begin{definition}
An \emph{unlikely propositional proof system} ---an upps for short--- is a pps $P$ such that for some polynomial $p$ we have that
\[
\forall^{\sf pps} Q\ P\geq_{p(x)} Q .
\]
\end{definition}

\begin{theorem}\label{theorem:superupps}
$\exists \mbox{\ upps } \Longleftrightarrow \  \exists \mbox{\ super pps\ }
\Longleftrightarrow \ {\sf coNP=NP}$
\end{theorem}

\begin{proof}
By a basic Theorem \ref{theorem:CookReckhow} we know that 
$\exists \mbox{\ super pps\ }\Leftrightarrow \ {\sf coNP=NP}$. We relate the existence of an upps to the existence of a super pps by actually proving that 
$P$ is an upps $\Leftrightarrow$ $P$ is super.
\medskip

``$\Rightarrow$''\ Let $P$ be an upps with corresponding polynomial $p$. It follows that $P$ is super. For, let $\tau$ be some tautology. 
Then\footnote{There is a subtle technicality here as to the representation of 
$P+\tau$. It is tempting to define the mapping $P+\tau$ (remember, a proof system is a mapping) to be the identity on $\tau$. By definition $P$ was defined on $\tau$ too. The value of $P(\tau)$ should now be given on some other input, etc. We shall not go into the details of this coding here and assume some canonical representation.} 
\[
P\geq_{p(x)}P + \tau , \ \ \mbox{ whence } \ \ P\vdash_{p(|\tau|)} \tau.
\]
\medskip

``$\Leftarrow$''\ Let $P$ be super with corresponding polynomial $p$. Then $P$ is also an upps. For, let $Q$ be an arbitrary pps. If $Q(\pi,\tau)$, then, by our assumption on pps's, we see that  
$|\tau| \leq |\pi |$. As $P$ is super, we can find $\pi'$ with $P(\pi', \tau)$ and $|\pi'|\leq p(|\tau|)$. By monotonicity of $p$, clearly $|\pi'| \leq p(|\pi|)$ and we see that $P$ indeed is an ufacop.
\end{proof}

It is only in this proof (proof of Theorem \ref{theorem:superupps}) that we need the assumption on an upps $P$ that $P(\pi,\tau)$ implies $|\pi|\leq|\tau|$.

We shall relate uppses to ufacops --\emph{unlikely fast consistency provers}-- which are an adaptation of facop's. Basically, the idea is that a ufacop is a uniform version of a facop. That is, we swap quantifiers. For a facop $S$ there is, for any poly-time theory $T$, a polynomial $p$ such that $||{\sf Con}_T(\underline{n})||_S\leq p(n)$. 

If we would simply define a ufacop $S$ to be such that there is a polynomial $p$ such that for any poly-time theory $T$ we have  
$||{\sf Con}_T(\underline{n})||_S\leq p(n)$, it would be easy to see that there are no ufacops. This is because the axiomatization of $T$ could be very very long, so that the length of ${\sf Con}_T(\underline{n})$ cannot be bounded.
So, we define a measure of the complexity of $T$ that will go into the definition of a ufacop.

\begin{definition}
If $R$ is a relation that is decidable in time $\mathcal{O}(|x|^l)$ we shall call $l$ the \emph{decision exponent} of $R$ and write 
$l={\sf DecExp}(R)$. 

If $T$ is a theory with a poly-time decidable set of axioms, we denote by ${\sf DecExp}(T)$ the decision exponent of the set of axioms of $T$.
\end{definition}
\begin{definition}
An \emph{unlikely fast consistency prover}, a \emph{ufacop} for short, is a 
$\Delta_1^b(S_2^1)$ axiomatizable theory $S$ such that there is a polynomial $p$ such that
\[
\forall^{\Delta_1^b(S_2^1)}T \ \forall x \ \ ||{\sf Con}_T(\underline{x})||_S\leq \, p(x^l).
\]
Here, $l:= {\sf DecExp}(T)$.
\end{definition}
Before we shall relate uppses to ufacops we first need some additional observations on coding techniques.
\begin{lemma}\label{lemma:polytranslation}
If $R$ is a poly-time relation with ${\sf DecExp}(R) =l$, then there is a series of propositional formulas $\rho_n$ such that for some polynomial independent of $R$ we have
\begin{enumerate}
\item
$a\in R \ \Leftrightarrow \ \tilde{a}\models \rho_n$ \ \ for $|a|{\leq}(n{+}1)$,

\item
$|\rho_n| = \mathcal{O}(p(n^l))$.

\end{enumerate}
 
 \end{lemma}

\begin{proof}
It is well known that a relation $R$ which is decidable in time 
$\mathcal{O}(n^l)$, has circuits $C_n$ of size linear in ${\sf time}\times{\sf space}$. Clearly, the space is bounded by the time, yielding $\mathcal{O}(n^{2l})=\mathcal{O}((n^l)^2)$. The circuits can be translated in the standard way to propositional formulas which are not much larger than the circuits. All this scaling by coding techniques can be collected in a polynomial $p$.
\end{proof}

From this lemma it follows that for any $\Delta^b_1(S^1_2)$ relation $R$, there is an $l'$ such that $\forall n\ |\rho_n|\leq p(n^{l'})$ for the $\rho_n$ and $p$ as in the lemma above. For the sake of readability we shall assume that $l=l'$. Alternatively one could define ${\sf DecExp}(R)$ to be this very $l'$.

\begin{lemma}\label{lemma:contranslation}
Let $T$ be a theory with a poly-time set of axioms with ${\sf DecExp}(T)=l$. There is a translation
${\langle}|\cdot|{\rangle}^m$ of specific $\Pi_1^b$ formulas into series of propositional formulas such that there is a fixed (independent of $T$) polynomial $q$ such that 
\[
|({\langle}|{\sf Con}_T(|x|)|{\rangle}^m)|\leq q(m^l).
\]
\end{lemma}

\begin{proof}
The formula ${\sf Con}_T(|x|)$ says 
$\forall \, |y|{<}|x|\ \neg {\sf Proof}_T(y,\ulcorner 0=1 \urcorner)$.
Here, \\
${\sf Proof}_T(y,\ulcorner 0=1 \urcorner)$ is as always saying that 
$y$ is a sequence (a proof) where some entries are axioms of $T$. We translate the $\Delta^b_1(S^1_2)$ formula ${\sf Axioms}_T(x)$ using Lemma 
\ref{lemma:polytranslation} and the rest in the structural way as mentioned in Lemma \ref{lemma:coding}. 
\end{proof}
Note that this translation ${\langle}|\cdot|{\rangle}^m$ still has all the structural properties as mentioned in Lemma \ref{lemma:coding}. We shall in the sequel refrain from distinguishing ${\langle}|\cdot|{\rangle}^m$ and $||\cdot||^m$.

\begin{theorem}\label{theorem:ufacop}
The following are equivalent.
\begin{enumerate}
\item \label{theorem:ufacop:item:ufacop}
$\exists$ ufacop

\item \label{theorem:ufacop:item:upps}
$\exists$ upps

\item \label{theorem:ufacop:item:super}
$\exists$ super pps

\item 
$\sf coNP=NP$

\end{enumerate}

\end{theorem}

\begin{proof}
In the light of Theorem \ref{theorem:superupps}, we only need to concentrate on \ref{theorem:ufacop:item:ufacop}.
First we show that \ref{theorem:ufacop:item:ufacop} $\Rightarrow$ 
\ref{theorem:ufacop:item:super} and then we shall show that 
\ref{theorem:ufacop:item:upps} $\Rightarrow$ \ref{theorem:ufacop:item:ufacop}.
\medskip

``\ref{theorem:ufacop:item:ufacop} $\Rightarrow$ 
\ref{theorem:ufacop:item:super}'' \ 
Suppose $S$ is a ufacop. We define $P_S$ as follows:
\[
\begin{array}{ll}
P_S(\pi, \tau)\ := & {\sf Proof}_S(\pi, {\sf Taut}(\tau)) \ \mbox{ or }\\
\ & ( \tau =1 \mbox{ and $\pi$ is not a proof in $S$ of ${\sf Taut}(\tau')$ for any $\tau'$)}
\end{array}
\]
We shall show that $P_S$ is super. Our proof is a simplification of the proof of the analog of this implication in Theorem \ref{theorem:facopiffopps}. Moreover, we keep track of the explicit polynomials here.

Via Theorem \ref{theorem:sigmacomplete} we get a polynomial $q(x)$ such that 
\[
\begin{array}{ll}
S\vdash \forall x\ (\neg {\sf Taut}(x) \to 
{\sf Pr}_{S_2^1}(q(|x|),\ulcorner \neg {\sf Taut}(\dot{x})\urcorner))
& \Leftrightarrow \\
S\vdash \forall x\ (\neg {\sf Taut}(x) \to
\exists \pi\ ( |\pi| {<} q(|x|) \wedge {\sf Proof}_{S_2^1}(\pi,\ulcorner \neg {\sf Taut}(\dot{x})\urcorner))) & 
\Rightarrow \\
S\vdash \forall x\ (\neg {\sf Taut}(x) \to
\exists \pi\ ( |\pi| {<} q'(|x|) \wedge {\sf Proof}_{S_2^1 + {\sf Taut}(\dot{x})}(\pi,\ulcorner 0=1
\urcorner))) & \ \\
\mbox{ for some polynomial $q'$ not so different from $q$ } & \Rightarrow \\
S\vdash \forall x\ (\forall \pi\ ( |\pi| {<} q'(|x|) \to \neg {\sf Proof}_{S_2^1 + {\sf Taut}(\dot{x})}(\pi,\ulcorner 0=1
\urcorner))
\to {\sf Taut}(x) 
) &\Rightarrow \\
S\vdash \forall x\ ({\sf Con}_{S_2^1 + {\sf Taut}(\dot{x})}(q'(|x|))
\to {\sf Taut}(x) 
) & \ \ \ \ \ (\dag)
\end{array}
\]
Now, as $S$ is an ufacop, there is a polynomial $p$ such that 
\[
||{\sf Con}_{S_2^1 + {\sf Taut}(\tau)}(q'(|\tau|))||_S \ \  \leq \ \ p((q'(|\tau|))^l) \ \ \ \ \ (\dag\dag)
\]
where $l={\sf DecExp}(S_2^1 + {\sf Taut}(\tau))$. Combining $(\dag)$  and $(\dag\dag)$ we get that
\[
S\vdash_{p'(|\tau|^l)} {\sf Taut}(\tau)
\]
for some polynomial $p'$. Note that $p'$ and $q'$ are independent of $\tau$. To conclude our argument we only need to see that 
$l={\sf DecExp}(S_2^1 + {\sf Taut}(\tau))$ is independent of $\tau$. 

However, to check whether $x$ is an axiom of $S_2^1 + {\sf Taut}(\tau)$, we have to check whether $x$ is an axiom of $S_2^1$ or whether $x={\sf Taut}(\tau)$ which is linear in $|x|$ (and so are the corresponding circuits). So, indeed, $l$ is independent on $\tau$ and $P_S$ is super.
\medskip

``\ref{theorem:ufacop:item:upps} $\Rightarrow$ \ref{theorem:ufacop:item:ufacop}'' \ 
So, we now prove $\exists$ upps $\Rightarrow$ $\exists$ ufacop. Suppose that $P$ is an upps with corresponding polynomial $p$. We claim that 
\[
S:=S^1_2 + {\sf RFN}(P)
\]
is an ufacop. To see this, we consider an arbitrary $\Delta^b_1(S^1_2)$ axiomatized theory $T$ with $l={\sf DecExp}$ and estimate $||{\sf Con}_T(\underline{n})||_S$.  The theory $T$ will be related to $P$ by defining 
\[
Q:= P+ \{  ||{\sf Con}_T(|x|)||^m \mid m< \omega\} .
\]
Note that, as we have a logarithm, indeed, ${\sf Con}_T(|x|)$ is a $\Pi_1^b$ formula and by Lemma \ref{lemma:contranslation} we get that 
\[
|(||{\sf Con}_T(|x|)||^m)| \leq q(m^l)
\]
for some polynomial $q$ independent of $T$. As $P$ is an upps we get 
\[
P\vdash_{p(q(m^l))}||{\sf Con}_T(|x|)||^m .
\]
By Theorem \ref{theorem:NPenumerable} we get some polynomial $r$, independent of $T$ such that 
\[
S\vdash_{r(m^l)} \exists y\ P(y,||{\sf Con}_T(|x|)||^m).
\]
As $S$ contains ${\sf RFN}(P)$, we can perform the following reasoning inside $S$. Note that the reasoning is uniform and not depending on particular properties of $T$ other than $l$.
\[
\begin{array}{lll}
\ \ \ \ \ \ \ \ \ \ (\star)
 & \ & \ \\
\exists y\ P(y,||{\sf Con}_T(|x|)||^m) & \Rightarrow & 
\mbox{by ${\sf RFN}(P)$}\\
{\sf Taut}(||{\sf Con}_T(|x|)||^m)& \Rightarrow & 
\mbox{by definition of ${\sf Taut}$}\\
\forall \,  |a| {\leq} (m{+}1) \ \ \tilde{a} \models ||{\sf Con}_T(|x|)||^m& \Rightarrow & 
\ \\
\widetilde{2^m}\models ||{\sf Con}_T(|x|)||^m& \Rightarrow & 
\ \mbox{ by Lemma \ref{lemma:coding}, Item \ref{lemma:coding:item:adequate}}\\
{\sf Con}_T(\underline{m}) & \ & \ \\
\ \ \ \ \ \ \ \ \ \ (\star\star)
 & \ & \ \\
\end{array}
\]
As we mentioned, this reasoning is not dependent on $T$ other than via\\ 
$|(||{\sf Con}_T(|x|)||^m)|$ and $|{\sf Con}_T(\underline{m})|$. Lemma
\ref{lemma:contranslation} takes care of the first. For the second, we shall use and assumption, namely that $|{\sf Con}_T(\underline{m})|$ is not much larger than $t(|m|^l)$ for some polynomial $t$ independent of $T$. This is not a strange assumption. 

When formalizing mathematics, at some stage, one should often exclude some pathological codings. In our case, $|{\sf Con}_T(\underline{m})|$ is only dependent on $|{\sf Axioms}_T(x)|$. By coding the small circuit that decides whether a number is a code of an axiom (see Lemma 
\ref{lemma:polytranslation}) in arithmetic, we get a short way of writing 
${\sf Axioms}_T(x)$.

If we put no restrictions on the way ${\sf Axioms}_T(x)$ is represented, it might very well consist of $10^{10^{99}}$ conjunctions of the short representation. One can even think of worse pathological codings.

Under our assumption, indeed, the reasoning between $(\star)$
and $(\star\star)$ can be performed in $S$ in a uniform way, whence for some polynomial $p'$ independent of $T$ we get that
\[
S\vdash_{p'(m^l)}{\sf Con}_T(\underline{m}).
\]
In other words $||{\sf Con}_T(\underline{m})||_S \leq p'(m^l)$ and $S$ is indeed a ufacop.
\end{proof}

\begin{question}
Under the assumption that ${\sf NP \neq coNP}$, is there an oracle relativized to which there is a facop which is not an ufacop?
\end{question}
It is clear that if the answer to this question is positive, then the existence of an opps really is (conditionally) weaker than ${\sf NP}={\sf coNP}$.
Buhrman et al. gave in \cite{Buhrman2000} an oracle under which no pps and a fortiori no facop does exist.

Verbitsky gave in \cite{Verbitsky1991} an oracle such that optimal proof systems exist, however still $\sf NE \neq coNE$, whence $\sf coNP\neq NP$.

\section{Lower bounds for facops}
Of course, having an equivalence of $\sf NP{=}coNP$ to the existence of a ufacop does not directly help to attack this problem: hard problems are never solved by reformulating them. As expected, problems related to ufacops and facops are likely to be extremely difficult. 

For example, it is not even known of specific weak theories like, for example, $S^1_2$ or even Robinson's arithmetic $Q$ that they are not a facop. In this section we shall present and reprove some well known results which are the best lower bound results known when it comes to facops.
Friedman and Pudl\'ak independently have shown the following theorem.

\begin{theorem}\label{pudlakfriedman}
For every poly-time axiomatizable 
theory $T$ extending $S^1_2$, there is a number $0 < \epsilon < 1$ such that
\[
n^{\epsilon} < || \cons_T(\numeral{n}) ||_T.
\]
\end{theorem}

\begin{proof}
The proof can also be found in \cite{Pudlak98}. Our proof is a bit sketchy. More details shall be given in Lemma \ref{lemm:mainidea} where the proof is milked further.

The proof proceeds by considering a fixed point $\delta(x)$ satisfying the
following equivalence.
\[
T \vdash \delta(x) \leftrightarrow \neg \Pr_T(x,\code{\delta(\dot{x})})
\]

Now, we reason in $T$.\\
Suppose $T \bewijs{x} \delta (\numeral{x})$, then, by 
Theorem \ref{theorem:sigmacomplete}, for some polynomial $f$ we
get $T \bewijs{f(x)}\Pr_T(\numeral{x},\code{\delta(\numeral{x})})$. This yields, combining with properties of the fixed-point, that 
$T \bewijs{g(x)}0=1$, for some function $g(x)=\oo{f(x)+x+\log(x)^{\oo{1}}}$.
By contraposition we get that
\begin{equation}\label{pudlakfriedman:equation:ConNaarDelta}
\cons_T(g(x)) \to \neg \Pr_T(x,\code{\delta(\dot{x})})
\end{equation}
Here ends our reasoning inside $T$. Note that, as $\delta$ was externally given, the $f$ and $g$ in this reasoning are actually also externally given.\medskip

Now, as $T$ is consistent, we get from 
\eqref{pudlakfriedman:equation:ConNaarDelta} that 
$\neg \Pr_T(x,\code{\delta(\dot{x})})$, whence
$||\delta(\numeral{x})||_T > x$. 
It is reasonable to assume that $x=\oo{f(x)}$, whence $g(x)=\oo{f(x)}$.

Again, using the provable fixed point properties of $\delta(x)$ we obtain
from \eqref{pudlakfriedman:equation:ConNaarDelta} that 
\[
||\cons_T(g(\numeral{x}))\to \delta(\numeral{x})||_T = \log(x)^{\oo{1}}
\] 
whence
\[
||\cons_T(g(\numeral{x}))||_T \geq ||\delta(\numeral{x})||_T - \log(x)^{\oo{1}}
\geq x -\log(x)^{\oo{1}}.
\]
As $g(x)=\oo{f(x)}$ we get that (the inverses of polynomials on positive numbers exist from a certain point on) for $x$ large enough
\[
\begin{array}{lll}
||\cons_T(\underline{x})||_T  & \geq & ||\cons_T(g^{-1}(g(\underline{x})))||_T \\
\ &\geq &
||\cons_T(f^{-1}(g(\underline{x})))||_T\\
\ &\geq & f^{-1}(x -\log(x)^{\oo{1}})\\
 \ & \geq & x^{\epsilon}.  
\end{array}
\]
Here,  $\frac{1}{\epsilon}$ is about the size of the degree of $f$, whence
$0<\epsilon<1$.
\end{proof}

\subsection{Variations}

Most likely, it is possible to use any variation of the proof of G\"odel's second incompleteness theorem to get Theorem \ref{pudlakfriedman}. In particular, one can consider the proof that uses a fixed point of
\[
\delta(x) \leftrightarrow \Pr_T(x,\code{\neg \delta (\dot{x})}) .
\]
Again, it is easy to see that $T \vdash_x \neg \delta (\numeral{x})$ yields a contradiction. An extra application of reflection is needed to show that
$T\nvdash_x \delta(\numeral{x})$. 
\medskip

\noindent
It is also possible to run the same argument with the following fixed point.
\[
\delta(x) \leftrightarrow \neg \Pr_T (h(x),\code{\delta(\dot{x})})
\] 
Of course, the representation of $h$ should not block the provable completeness for $\Sigma^0_1$ (or $\exists \Sigma^b_1$) sentences that is needed in the argument. 
The function $h$ must thus be $\Sigma_1^0$ representable. In other words, $h$ should be a recursive function.
In the light of Question \ref{epsilon}, this fixed point only makes things worse.

However, this fixed point gives rise to true statements with very long proofs. This shall be exploited later on. Therefore it is worth while to restate some easy properties of this fixed point. We shall require that $h$ be some provably unbounded (that is, goes provably to infinity) recursive function.

\begin{fact}\label{fact:LongProofs}
Let $S$ be a sound theory and $\delta$ such that $S\vdash \delta(x) \leftrightarrow \neg \Pr_S (h(x),\code{\delta(\dot{x})})$. Then,
\begin{enumerate}
\item
$||\delta(\numeral{n})||_S > h(n)$

\item
$\mathbb{N} \models \forall n \ \delta (n)$

\item
$\forall n \  S \vdash \delta (\numeral{n})$

\item\label{bewijsbaarequiaanfalsum}
$S\nvdash \forall x\ \delta(x)$

\end{enumerate}
\end{fact}
\noindent
These facts are pretty easy to verify. At \ref{bewijsbaarequiaanfalsum} the provable unboundedness of $h$ is used to see that 
$S\vdash \forall x\ \delta(x) \leftrightarrow {\sf Con}(S)$.
Now, using these facts, we can give easy proofs of the following two well known propositions.
\begin{proposition}\label{predlogica}
For any recursive function $h$, there exists a series of provable predicate logical tautologies $\varphi_n$ of which the length of proofs in predicate logic are not bounded by $h(|\varphi_n|)$.
\end{proposition}

\begin{proof}
Take a strong enough finitely axiomatized arithmetic theory, for example
\isig{1}. Consider 
\[
\isig{1} \vdash \delta(x) \leftrightarrow \neg \Pr_{\isig{1}} (h(x),\code{\delta(\dot{x})}).
\]
Then $\bigwedge \isig{1}\to \delta(\numeral{n})$ suffices.
\end{proof}

\begin{proposition}
There is an explicit series of provable predicate logical tautologies $\psi_n$ whose proofs are not bounded by any recursive function.
\end{proposition}

\begin{proof}
By diagonalization from Proposition \ref{predlogica}.
\end{proof}
\noindent

\section{RE facops do not exist}\label{section:ReFacopsDoNotExist}

The existence of a facop or a ufacop is very counter-intuitive. However, as is to be expected, every attempt to prove the non-existence fails. In this section we shall present such an attempt by dropping the requirement that the theories for which a facop should have short proofs be poly-time decidable. 

So, in this section, we consider sound theories with an RE axiomatization. For this class of theories we can show that there is no "strongest theory" $S$ having short proofs for finite consistency statements of any other RE theory.
The final result is stated in Theorem \ref{theo:REpairs}. Actually the result is quite strong. It says that for any theory $S$, there is a theory $T$ whose proofs in $S$ of its consistency statements have non-recursive lengths.

The idea of the proof is by generalizing the proof of Theorem 
\ref{pudlakfriedman} and Fact \ref{fact:LongProofs}. First we state a lemma that articulates some conditions on $S$ and $T$ under which
$||{\sf Con}_T(\underline{x})||_S\geq h(x)$. We have chosen $S$ as to refer to \emph{slow}. The next two lemmata tell us how to construct, given an $S$, a theory $T$ such that the conditions hold.

\begin{lemma}\label{lemm:mainidea}
Let $S$ and $T$ be consistent RE theories containing $S_2^1$. Let $\delta(x)$ be such that 
\[
S^1_2\vdash \forall x\ (\delta (x) \leftrightarrow 
\neg \Pr_S(h(x), \code{\delta (\dot{x})})) 
\]
for a certain recursive $h$ with $h = \ohm{x}$.
Furthermore, let $S$ and $T$ be such that $T$ has speed-up over $S$ in 
the following sense.
\renewcommand{\theenumi}{(\roman{enumi})}
\renewcommand{\labelenumi}{\theenumi}
\begin{enumerate}
\item \label{item:speedup1}
$S \bewijs{h(x)} \delta (\numeral{x}) \ \ \Rightarrow \ \  
T\bewijs{\oo{x}} \Pr_S(h(\numeral{x}), \code{\delta (\numeral{x})})$,

\item \label{item:speedup2}
$S \bewijs{h(x)} \delta (\numeral{x}) \ \ \Rightarrow \ \  
T\bewijs{\oo{x}} \delta (\numeral{x})$.

\end{enumerate}
\noindent
Moreover, let \ref{item:speedup1} and \ref{item:speedup2} be formalizable in 
$S$. Then it holds that
\[
|| \cons_T (\numeral{x})||_S \geq h(\mathcal{O}(x)) .
\]
\end{lemma}

\begin{proof}
Reason in $S$. Suppose that 

\begin{equation}\label{eq:lemm:assumption}
S\bewijs{h(x)}\delta(\numeral{x}).
\end{equation}
Then, by Assumption \ref{item:speedup1}, we get
\begin{equation}
T\bewijs{\oo{x}}\Pr_S(h(\numeral{x}),\delta(\numeral{x})).
\label{eq:lemm:speedup}
\end{equation}
\noindent
Combining \eqref{eq:lemm:assumption} and \ref{item:speedup2}, we also get
\[
T\bewijs{\oo{x}} \delta (\numeral{x}) .
\]
As the fixed point equation is also provable in $T$, i.e.
\[
T\bewijs{\oo{1}} \forall x\ (\delta(x)   \leftrightarrow \neg 
\Pr_S(h(x), \code{\delta(\dot{x})})    ),
\]
we get
\[
T\bewijs{\oo{x}+\log(x)^{\oo{1}}} \neg \Pr_S (h(\numeral{x}), \code{\delta(\numeral{x})}).
\]
Combining this with \eqref{eq:lemm:speedup} we obtain
\[
T\bewijs{\oo{x}+\log(x)^{\oo{1}}} 0=1 .
\]
\ 
\medskip

\noindent
We now no longer reason in $S$. Considering the above reasoning, together with the fact that $S$ is sound and $T$ is consistent, we see that
\begin{equation}\label{eq:largeproofs}
|| \delta(\numeral{x}) ||_S \geq h(x).
\end{equation}
Also, from the above reasoning, we have 
\[
S\vdash \Pr_S(h(x), \delta (x)) \to \neg \cons_T(g(x))
\]
for some function $g(x)=\oo{x + \log(x)^{\oo{1}}}$. Consequently also
\[
S \vdash \cons_T(g(x)) \to \delta(x) \ \ \ \ \ (\leftrightarrow 
\neg \Pr_S(h(x), \delta(x))),
\]
and we get that
\[
|| \cons_T(g(\numeral{x})) \to \delta(\numeral{x}) ||_S = \log(x)^{\oo{1}}.
\]
This implies
\[
|| \cons_T(g(\numeral{x})) ||_S \geq || \delta (\numeral{x})||_S - \log(x)^{\oo{1}}.
\]
Because $g(x)=\oo{x}=\oo{h(x)}$, by \eqref{eq:largeproofs} we obtain the required result, that is,

\[
|| \cons_T(\numeral{x}) ||_S\geq h(\mathcal{O}(x)).
\]
\end{proof}
\noindent
The next lemma provides an approach so that we can concentrate on Item 
\ref{item:speedup1}.

\begin{lemma}\label{lemm:reflectionisgood}
Let $S$ and $T$ be (sound \& RE) such that
\[
\{  \forall \vec{x} \ (\Box_S \varphi(\dot{\vec{x}}) \to \varphi(\vec{x}))\} 
\subseteq \mbox{The axioms of $T$}
\]
and that moreover (verifiably in $S$)
\begin{enumerate}
\item\label{lemm:item:rig}
$S\bewijs{h(x)} \delta(\numeral{x}) \Rightarrow T\bewijs{\oo{x}}
\Pr_S(h(\numeral{x}),\delta(\numeral{x}))$
\end{enumerate}
for some fixed formula $\delta(x)$, then it holds (verifiably in $S$) that
\[
S\bewijs{h(x)}\delta(\numeral{x}) \Rightarrow T \bewijs{\oo{x}} \delta(\numeral{x}).
\]
\end{lemma}

\begin{proof}
(Reason in $S$.) Suppose that $S\bewijs{h(x)}\delta(\numeral{x})$. 
Because of \ref{lemm:item:rig}, we get that 
\[
T\bewijs{\oo{x}}\Pr_S(h(\numeral{x}), \delta(\numeral{x})),
\]
hence also
\[
T\bewijs{\oo{x}} \Box_S \delta(\numeral{x}).
\]
Adding just one more line to the $T$-proof consisting of the axiom 
$\Box_S \delta(\numeral{x}) \to \delta(\numeral{x})$ gets us 
the required
\[
T\bewijs{\oo{x}} \delta(\numeral{x}),
\]
as the number of symbols in
$\Box_S \delta(\numeral{x}) \to \delta(\numeral{x})$ 
is just $\oo{\log(x)}$.
\end{proof}
Note that this proof makes no further assumptions on the nature of 
$\delta(x)$. For the particular $\delta(x)$ we are interested in, it would suffice to demand that $T \supseteq {\sf Con(S)}$.

\begin{lemma}\label{lemm:REpair}
Let $S$ be a given sound \& RE theory. Let $S'$ be defined so that its axioms are precisely the theorems of $S$. Next, define $T$ so that its axioms are the axioms of $S'$ together with 
$\{  \forall \vec{x}\  (\Box_S \varphi(\dot{\vec{x}}) \to \varphi(\vec{x}))\}$.
Then, $S$ and $T$ satisfy  \ref{item:speedup1} and \ref{item:speedup2}
of Lemma \ref{lemm:mainidea}.
\end{lemma}

\begin{proof}
Note that $S=S'$, whence $S'$ and $T$ are also sound RE theories.
Reason in $S$, and suppose that
$S\bewijs{h(x)} \delta(\numeral{x})$. Then also 
$S\vdash \Pr_S(h(\numeral{x}), \delta (\numeral{x}))$.
Notice that the length of $\Pr_S(h(\numeral{x}), \delta(\numeral{x}))$
is \oo{log(x)} so certainly $\oo{x}$, whence
\[
T\bewijs{\oo{x}}  \Pr_S(h(\numeral{x}), \delta(\numeral{x})).
\]
Lemma \ref{lemm:reflectionisgood} now yields the desired result.
\end{proof}

Note that the construction in Lemma \ref{lemm:REpair} works simultaneously for all recursive functions.
Thus, putting things together, we have now shown the following theorem, as announced at the beginning of this section.

\begin{theorem}\label{theo:REpairs}
For any sound RE theory $S$ 
there exists  another sound RE theory $T$ for which
for any   recursive function $h$
\[
|| \cons_T(\numeral{x})||_S \geq h(\mathcal{O}(x)).
\]
\end{theorem}

\begin{proof}
For any such 
theory $S$, apply the construction as in Lemma \ref{lemm:REpair} to obtain a theory $T$ so that Lemma \ref{lemm:mainidea} yields
the required result.
\end{proof}
\begin{question}
Can Theorem \ref{theo:REpairs} also be proved for theories with a primitive recursive set of axioms? Which is the weakest class of theories for which Theorem \ref{theo:REpairs} can be proved?
\end{question}

\section{Speculations on poly-time diagonalizations}
Clearly, a theorem like Theorem \ref{theo:REpairs} can not be proved in full generality for poly-time theories. This is due to the observation we made before that $||{\sf Con}_T(\underline{x})||_S\leq f(x)$ for some $f$ which is exponential in in $x$. Of course, this observation hinges on the fact that it is poly-time decidable that an axiom of $T$ is indeed an axiom of $T$. And thus, by Theorem \ref{theorem:NPenumerable} the axiomhood has a short proof in $S$.

Having this in mind we immediately see why the proof 
of Theorem \ref{theo:REpairs} does not carry over to the setting of poly-time theories: If one starts out with a theory $S$ with a
$\Delta^b_1(\sonetwo)$ axiomatization, the trick in Lemma \ref{lemm:REpair} will
yield a genuinly $\Sigma^0_1$ axiomatized theory $T$.

One could think of defining the axioms of $S$ consisting of those theorems having a proof in some logarithmically short interval $[a,b]$ which is not too far away from the theorem. However, this is the same problem as we started with: given a provable formula, look for a short proof.

The conditions in Lemma \ref{lemm:mainidea} are formulated in quite a general way. A more promising way to obtain lower bounds for facops would be to look for other fixed points such that given a theory $S$, one can define a theory $T$ such that conditions \ref{item:speedup1} and \ref{item:speedup1} of  Lemma \ref{lemm:mainidea} are satisfied for this fixed point. 

The following conjecture does not seem fully unfeasible.

\begin{conjecture}\label{conj:desirable}
For every sound $\Delta^b_1(S^1_2)$ theory $S$ and for every $l\in \omega$, there exists a sound $\Delta_1^b(S^1_2)$ theory $T$ such that
\[
||\cons_T(\numeral{x})||_S > x^l .
\]
\end{conjecture}
It is clear that Conjecture \ref{conj:desirable} 
is a desirable result as it is just one step away of the required
\[
\exists^{\Delta^b_1(S^1_2)} S\, \forall^{\Delta^b_1(S^1_2)} T\, \forall l\, \  
||\cons_T(\numeral{x}) ||_S > x^l.
\]
And this last step suggests some compactness or diagonalization argument. However, poly-time diagonalization seems to be the
hard problem at the core of the ${\sf P} \neq {\sf NP}$-problem.

We would like to conclude this paper by an easy but interesting observation. Mathematical practice has proved that it is very hard to find strong lower bounds for classical propositional logic. Actually the state of the art is still stuck at a quadratic lower bound. The following observation might be an explanation for this fact.

\begin{observation}
If optimal proof systems do exist, then any poly-time recognizable sequence of tautologies has  polynomially bounded proofs. Hence, if moreover $\sf coNP\neq NP$, any hard tautology is not poly-time recognizable in this case.
\end{observation}

\paragraph{Acknowledgments}
First of all, I would like to thank Pavel Pudl\'ak for help, questions, and improvements. Next, there are many other people that I would like to mention here: Christian Bennet, Peter van Emde Boas, Harry Buhrman, Rosalie Iemhoff, Emil Je{\v{r}}{\'a}bek, Jan Kraj\'{\i}\v{c}ek, Neil Thapen, Leen Torenvliet, Alan Skelley, V{\'\i}t\v ezslav \v Svejdar, Oleg Verbitsky, and Albert Visser. 


\end{document}